\documentclass{amsart}

\usepackage{amsfonts, amsmath, amscd}
\usepackage[psamsfonts]{amssymb}

\usepackage{amssymb}

\usepackage{pb-diagram}

\usepackage[all,cmtip]{xy}

\usepackage[usenames]{color}

\newtheorem{theorem}{Theorem}[section]

\newtheorem{proposition}[theorem]{Proposition}

\newtheorem{fact}[theorem]{Fact}

\numberwithin{equation}{section}

\newcommand{\NN}{\mathbb{N}}

\newcommand{\w}{\omega}

\newcommand{\TTT}{\mathcal{T}}

\newcommand{\VV}{\mathbb{V}}

\newcommand{\Nn}{\mathcal{N}}

\newcommand{\AAA}{\mathcal A}
\newcommand{\IR}{\mathbb{R}}

\newcommand{\aaa}{\mathbf{a}}

\newcommand{\Ss}{\mathbb{S}}

\newcommand{\Ff}{\mathfrak{F}}

\newcommand{\VVs}{\VV_{\mathbf{s}}}

\renewcommand{\phi}{\varphi}

\newcommand{\conv}{\mathrm{conv}}

\input xy
\xyoption{all}

\title[The free locally convex space $L(\mathbf{s})$ is not a Mackey space]{The free locally convex space $L(\mathbf{s})$ over a convergent sequence $\mathbf{s}$ is not a Mackey space}

\author[S.~Gabriyelyan]{Saak Gabriyelyan}
\address{Department of Mathematics, Ben-Gurion University of the
Negev, Beer-Sheva, P.O. 653, Israel}
\email{saak@math.bgu.ac.il}



\begin{document}

\begin{abstract}
We show that the free locally convex space $L(\mathbf{s})$ over a convergent sequence $\mathbf{s}$ is not a Mackey space. Consequently $L(\mathbf{s})$ is not a Mackey group that answers negatively a question posed in \cite{Gab-Mackey}.
\end{abstract}

\maketitle

%%%%%%%%%%%%%%%%%%%%%%%%%%%
%%%%%%%%%%%%%%%%%%%%%%%%%%%
%%%%%%%%%%%%%%%%%%%%%%%%%%%
%%%%%%%%%%%%%%%%%%%%%%%%%%%

\section{Introduction}

%%%%%%%%%%%%%%%%%%%%%%%%%%%
%%%%%%%%%%%%%%%%%%%%%%%%%%%
%%%%%%%%%%%%%%%%%%%%%%%%%%%
%%%%%%%%%%%%%%%%%%%%%%%%%%%

Let $(E,\tau)$ be a locally convex space. A locally convex vector topology $\nu$ on $E$ is called {\em compatible with $\tau$} if the spaces $(E,\tau)$ and $(E,\nu)$ have the same topological dual space. The classical  Mackey--Arens theorem states that for every locally convex space $(E,\tau)$ there exists the finest locally convex vector space topology $\mu$ on $E$ compatible with $\tau$. The topology $\mu$ is called the {\em Mackey topology} on $E$ associated with $\tau$, and if $\mu=\tau$, the space $E$ is called a {\em Mackey space}.

For an abelian topological group $(G,\tau)$ we denote by $\widehat{G}$ the group of all continuous characters of $(G,\tau)$.
Two topologies  $\mu$ and $\nu$ on an abelian group $G$  are said to be {\em compatible } if $\widehat{(G,\mu)}=\widehat{(G,\nu)}$.
Being motivated by the Mackey--Arens Theorem and following \cite{CMPT},  a locally quasi-convex  abelian group $(G,\mu)$ is called a {\em Mackey group} if for every locally quasi-convex group topology $\nu$ on $G$ compatible with $\tau$  it follows that $\nu\leq\mu$.
It follows from Proposition 2.5 of \cite{Gab-Mackey} that if a locally convex space $E$ is not Mackey, then $E$ also is not a Mackey group.

In \cite[Question 4.3]{Gab-Mackey}, we posed the following question: Is the free locally convex space $L(\mathbf{s})$ over a convergent sequence $\mathbf{s}$ a Mackey group? In this note we answer this question in the negative in a stronger form:
\begin{theorem} \label{t:L(s)-Mackey}
The locally convex space $L(\mathbf{s})$ is not a Mackey space. Consequently  $L(\mathbf{s})$ is not a Mackey group.
\end{theorem}
We prove Theorem \ref{t:L(s)-Mackey} in the next section.

%%%%%%%%%%%%%%%%%%%%%%%%%%%%%%
%%%%%%%%%%%%%%%%%%%%%%%%%%%%%%
%%%%%%%%%%%%%%%%%%%%%%%%%%%%%%
%%%%%%%%%%%%%%%%%%%%%%%%%%%%%%
%%%%%%%%%%%%%%%%%%%%%%%%%%%%%%

\section{Proof of Theorem \ref{t:L(s)-Mackey}} \label{sec:Mackey-F0}

%%%%%%%%%%%%%%%%%%%%%%%%%%%%%%
%%%%%%%%%%%%%%%%%%%%%%%%%%%%%%
%%%%%%%%%%%%%%%%%%%%%%%%%%%%%%
%%%%%%%%%%%%%%%%%%%%%%%%%%%%%%
%%%%%%%%%%%%%%%%%%%%%%%%%%%%%%

%%%%%%%%%%%%%%%%%%%%%%%%%%%%%%%%%%%%%%%%%%%%%%%%
%%%%%%%%%%%%%%%%%%%%%%%%%%%%%%%%%%%%%%%%%%%%%%%%
%%%%%%%%%%%%%%%%%%%%%%%%%%%%%%%%%%%%%%%%%%%%%%%%

Set $\NN:=\{ 1,2,\dots\}$ and $\w=\{ 0\} \cup\NN$.
Denote by $\mathbb{S}$ the unit circle group and set $\Ss_+ :=\{z\in  \Ss:\ {\rm Re}(z)\geq 0\}$.
Let $G$ be an abelian topological group.  If $\chi\in \widehat{G}$, it is considered as a homomorphism from $G$ into $\mathbb{S}$.
A subset $A$ of $G$ is called {\em quasi-convex} if for every $g\in G\setminus A$ there exists   $\chi\in \widehat{G}$ such that $\chi(x)\notin \Ss_+$ and $\chi(A)\subseteq \Ss_+$.
%If $A\subseteq G$ and $B\subseteq \widehat{G}$ set
%\[
%A^\triangleright :=\{ \chi\in \widehat{G}: \chi(A)\subseteq \Ss_+\}, \quad B^\triangleleft:=\{ g\in G: \chi(g)\in\Ss_+ \; \forall \chi\in B\}.
%\]
%Then $A$ is quasi-convex if and only if $A^{\triangleright\triangleleft}=A$.
An abelian topological group $G$ is called {\em locally quasi-convex} if it admits a neighborhood base at the neutral element $0$ consisting of quasi-convex sets. It is well known that the class of locally quasi-convex abelian groups is closed under taking products and subgroups. The dual group $\widehat{G}$ of $G$ endowed with the compact-open topology is denoted by $G^{\wedge}$.

Let $G$ be an abelian topological group and let $\mathcal{N}(G)$ be the filter of all open neighborhoods at zero of $G$. Denote by $G^\mathbb{N}$ the group of all sequences $(x_n)_{n\in\NN}$. %The subgroup of $G^\mathbb{N}$ of all sequences eventually equal to zero we denote by $G^{(\mathbb{N})}$.
It is easy to check  that the collection $\{ V^\mathbb{N}: V \in \mathcal{N}(G)\}$ forms a base at $0$ for a group topology in $G^\mathbb{N}$. This topology  is called {\em the uniform topology} and is denoted by $\mathfrak{u}$. Following \cite{DMPT}, denote by $c_0 (G)$ the following subgroup of $G^\mathbb{N}$
\[
c_0 (G) := \left\{ (x_n)_{n\in\mathbb{N}} \in G^\mathbb{N} : \; \lim_{n} x_n = 0 \right\}.
\]
The uniform group topology on $c_0 (G)$ induced from $(G^\mathbb{N}, \mathfrak{u})$ we denote by  $\mathfrak{u}_0$ and set $\Ff_0 (G) :=(c_0 (G), \mathfrak{u}_0)$. Clearly, the group $\mathfrak{F}_0(\mathbb{R})$ coincides with the classical Banach space $c_0$.

For  a subset $A$ of a topological vector space $E$, we denote the convex hull of $A$ by  $\conv(A)$, so
\[
\conv(A) :=\left\{ \lambda_1 a_1 +\cdots+ \lambda_n a_n : \; \lambda_1,\dots,\lambda_n \geq 0, \, \sum_{i=1}^n \lambda_i =1, \, a_1,\dots,a_n\in A, \, n\in\NN\right\}.
\]
If $n\in\NN$, we set $(1)A:=A$ and $(n+1)A:= (n)A+A$.

Let $X$ be a  Tychonoff space and $e\in X$ a distinguished point. Recall that the {\em  free locally convex space} $L(X)$ over $X$ is a pair consisting of a locally convex space $L(X)$ and  a continuous map $i: X\to L(X)$ such that every  continuous map $f$ from $X$ to a locally convex space  (lcs) $E$ gives rise to a unique continuous linear operator ${\bar f}: L(X) \to E$  with $f={\bar f} \circ i$. For every Tychonoff space $X$, the free locally convex space $L(X)$ exists, is unique up to isomorphism of locally convex spaces,  and $L(X)$ is algebraically the free vector space on $X$.

It is more convenient for our purpose to consider Graev free locally convex spaces instead of free locally convex spaces.
Recall that the {\em Graev free locally convex space} $L_G(X)$ over $X$ is a pair consisting of a locally convex space $L_G(X)$ and  a continuous map $i: X\to L(X)$ such that $i(e)=0$ and every  continuous map $f$ from $X$ to a locally convex space  $E$ gives rise to a unique continuous linear operator ${\bar f}: L_G(X) \to E$  with $f={\bar f} \circ i$. Also the Graev free locally convex space $L_G(X)$ exists, is unique up to isomorphism of locally convex spaces,  and is independent of the choice of $e$ in $X$. Further, $L_G(X)$ is algebraically the free vector space on $X\setminus \{ e\}$.
We shall use the following fact which can be obtained exactly as in \cite{Gra} or Proposition 2.10 in \cite{GM}: For every Tychonoff space $X$, one holds
\[
L(X)= \IR \oplus L_G(X).
\]
Therefore $L(X)$ is a Mackey space if and only if $L_G(X)$ is a Mackey space, see Corollary 8.7.5 and Theorem 8.8.5 of \cite{Jar}.

Let $\mathbf{s}=\{ \frac{1}{n}\}_{n\in\NN} \cup \{0\}$ be a convergent sequence with the usual topology induced from $\IR$. Set $\mathbf{s}^\ast :=\mathbf{s}\setminus \{ 1\}$, so $\mathbf{s}$ is the disjoint union of $\{ 1\}$ and $\mathbf{s}^\ast$. Then
\begin{equation} \label{equ:L(s)=Lg(s)}
L_G(\mathbf{s})=L_G(\{ 1\})\oplus L_G(\mathbf{s}^\ast)= \IR \oplus L_G(\mathbf{s}^\ast)=\IR \oplus L_G(\mathbf{s})=L(\mathbf{s}).
\end{equation}

Dealing with the concrete case $X=\mathbf{s}$ it is convenient and useful to describe the space $L_G(\mathbf{s})$ and its topology $\pmb{\nu}_{\mathbf{s}}^G$ explicitly that we do below. We shall  use this description also in forthcoming papers.
The free vector space $\VV_\mathbf{s}$ over $\mathbf{s}$ is just the direct sum $\IR^{(\NN)}$ of $\aleph_0$-many copies of the reals $\IR$. Denote by $\pi_n$ the projection of  $\IR^{(\NN)}$ onto the $n$th coordinate, and we denote  by $\mathbf{0}$ the zero vector of $\VV_\mathbf{s}$. For every $n\in\NN$, set $e_n=(0,\dots,0,1,0,\dots) \in \VV_\mathbf{s}$, where $1$ is placed in position $n$. Below we also find the maximal vector topology $\pmb{\mu}_{\mathbf{s}}^G$ and the maximal locally convex vector topology $\pmb{\nu}_{\mathbf{s}}^G$ on $\VV_\mathbf{s}$ in which the sequence $\mathbf{s}$ converges to $\mathbf{0}$. Then (see \cite{GM}) the topological vector space $\VV_G(\mathbf{s}):=(\VV_\mathbf{s}, \pmb{\mu}_{\mathbf{s}}^G)$ is called the {\em Graev free topological vector space} over $\mathbf{s}$, and the space $L_G(\mathbf{s}):=(\VV_\mathbf{s}, \pmb{\nu}_{\mathbf{s}}^G)$ is the  Graev free locally convex space on $\mathbf{s}$.
%Note that $\VV_G(\mathbf{s})$ and $L_G(\mathbf{s})$ are topologically isomorphic to the free topological vector space $\VV(\mathbf{s})$ and the free locally convex space $L(\mathbf{s})$ over $\mathbf{s}$, respectively, see Proposition 2.10 of \cite{GM}.

Below we describe the topologies $\pmb{\mu}_{\mathbf{s}}^G$ and $\pmb{\nu}_{\mathbf{s}}^G$.
Let $\AAA$ be the family of all sequences $\aaa=(a_n)_{n\in\NN}$ such that  $a_n >0$ for all $n\in\NN$ and  $\lim_{n} a_{n} =\infty$. Denote by $\mathcal{S}_\mathbf{s}$ the family of all subsets of $\VVs$ of the form
\[
S(\aaa):= \bigcup_{n\in\NN} \left( (-a_n,a_n) \times \prod_{j\in\NN ,\ j\not= n} \{ 0\} \right), \quad \mbox{ where } \aaa\in\AAA.
\]
For every sequence $\mathrm{S}=\{ S(\aaa_k) \}_{k\in\w}\subset \mathcal{S}_\mathbf{s} $, we put
\[
U(\mathrm{S})=\sum_{k\in\w} S(\aaa_k) := \bigcup_{k\in\w} \big(S(\aaa_0) +S(\aaa_1) +\cdots +S(\aaa_k)\big).
\]
If $\aaa_k =(a_{n,k})_{n\in\NN}$ for every $k\in\w$, then $v\in \conv\big( U(\mathrm{S})\big)$ if and only if there are $m\in\NN$ and $r\in\w$ such that $v$ has the form
\[
v=\sum_{i=1}^m \lambda_i \left( \sum_{j=0}^r t_{j,i} e_{l(j,i)} \right),
\]
where
\begin{enumerate}
\item[(a)] $\lambda_1,\dots,\lambda_m \in (0,1]$ and $\lambda_1+\cdots +\lambda_m =1$;
\item[(b)] $l(j,i)\in\NN$ for every $0\leq j\leq r$ and $1\leq i\leq m$;
\item[(c)] $|t_{j,i}| < a_{l(j,i), j}$  for every $0\leq j\leq r$ and $1\leq i\leq m$.
\end{enumerate}
Denote by $\Nn_\mathbf{s}$ the family of all subsets of $\VVs$ of the form $U(\mathrm{S})$.
\begin{proposition} \label{p:Ascoli-Finest}
The family $\Nn_\mathbf{s}$  forms a base at $\mathbf{0}$ for $\pmb{\mu}_{\mathbf{s}}^G$. Thus the family
\[
\widehat{\Nn}_\mathbf{s} :=\{ \conv(V) : V\in \Nn_\mathbf{s} \}
\]
is a base at $\mathbf{0}$ for  $\pmb{\nu}_{\mathbf{s}}^G$. In particular, for every neighborhood $U$ of zero in $\pmb{\nu}_{\mathbf{s}}^G$ and for every $t\in\NN$ and each $a>0$, there exists  $q\in\NN$ such that every vector of the form
\[
v=\lambda_1 a e_{m_1}+\cdots + \lambda_t a e_{m_t},  \mbox{ where } q<m_1<\cdots<m_t \mbox{ and } \lambda_1,\dots,\lambda_t\in[-1,1],
\]
belongs to $U$.
\end{proposition}

\begin{proof}
First we check that the family $\Nn_\mathbf{s}$ satisfies the following conditions:
\begin{enumerate}
\item[(a)] for each $U(\mathrm{S})\in\Nn_\mathbf{s}$ there is a $U(\mathrm{S}')\in\Nn_\mathbf{s}$ with $U(\mathrm{S}')+ U(\mathrm{S}')\subseteq U(\mathrm{S})$;
\item[(b)] for each $U(\mathrm{S})\in\Nn_\mathbf{s}$  there is a $U(\mathrm{S}')\in\Nn_\mathbf{s}$ such that $\lambda U(\mathrm{S}')\subseteq U(\mathrm{S})$ for all $|\lambda|\leq 1$;
\item[(c)] for each $U(\mathrm{S})\in\Nn_\mathbf{s}$ and every $\mathbf{x}\in \VVs$, there is $n\in\NN$ such that $\mathbf{x}\in  nU(\mathrm{S})$;
\item[(d)] $\bigcap \Nn_\mathbf{s} =\{ 0\}$.
\end{enumerate}
Indeed, let $\mathrm{S}=\{ S(\aaa_k) \}_{k\in\w}$, where $\aaa_k =(a_{n,k})_{n\in\NN} \subseteq (0,\infty)$. Set
\[
\mathrm{S}'=\{ S(\aaa'_k) \}_{k\in\w}, \mbox{ where } \aaa'_k := \left( \frac{1}{2}\min\{ a_{n,0},\dots,a_{n,2k+1}\} \right)_{n\in\NN}.
\]
Then $S(\aaa'_i)+S(\aaa'_i) \subseteq S(\aaa_{2i}) + S(\aaa_{2i+1})$ for every $i\in\w$. Therefore
\[
(2)\big( S(\aaa'_0)+\cdots + S(\aaa'_k) \big) \subseteq S(\aaa_0)+\cdots + S(\aaa_{2k+2}) \big) \mbox{ and } U(\mathrm{S}')+ U(\mathrm{S}')\subseteq U(\mathrm{S}),
\]
that proves (a). Items (b)-(d) are trivial. So, by \cite[\S 15, 2(1)]{Kothe}),   $\Nn_\mathbf{s}$  forms a base for some vector topology $\tau$ on $\VVs$.  Moreover,  $e_n \to \mathbf{0}$ in $\tau$, since $e_n \in S(\aaa_0)$ for all sufficiently large $n$.

To show that $\tau=\pmb{\mu}_{\mathbf{s}}^G$, fix an arbitrary vector topology $\tau'$ on $\VVs$ in which $e_n \to \mathbf{0}$ and let $V\in\tau'$.  Choose inductively a sequence of symmetric open circled neighborhoods $\{ V_k\}_{k\in\w}$ of $\mathbf{0}$ in $\tau'$ such that $V_0 +V_0\subseteq V$ and $V_{k+1}+V_{k+1}\subseteq V_{k}$ for each $k\in\w$, so $\sum_{k\in\w} V_k \subseteq V$. For every $k\in\w$ and each $n\in\NN$, set
\[
a_{nk}:=\sup\left\{ \lambda \in (0,n] : \mbox{ if } |x|<\lambda, \mbox{ then } x\in V_k \cap \left( \IR_n \times \prod_{j\in\NN ,\ j\not= n} \{ 0\} \right) \right\}.
\]
We show that $\lim_n a_{nk} =\infty$ for every $k\in\w$. Indeed, if $a_{n_i k}<A$ for some sequence $n_1 <n_2<\dots$ and a positive number $A$, then $e_{n_i} \not\in \frac{1}{A} V_k$ for every $i\in\NN$ (we recall that $V_k$ is circled). Hence $e_n\not\to 0$ in $\tau'$, a contradiction. For every $k\in\w$, we set $\aaa_k =(a_{nk})_{n\in\NN}$ and let $\mathrm{S}=\big\{ S(\aaa_k)\big\}$, so $\mathrm{S}\subset \mathcal{S}_\mathbf{s}$. Then, by construction, $U(\mathrm{S})\subseteq V$. Thus $\tau' \subseteq\tau$ and $\tau=\pmb{\mu}_{\mathbf{s}}^G$.

The fact that the family $\widehat{\Nn}_\mathbf{s}$ is a base for $\pmb{\nu}_{\mathbf{s}}^G$ follows from the definitions of $\pmb{\mu}_{\mathbf{s}}^G$ and $\pmb{\nu}_{\mathbf{s}}^G$ (cf. Proposition 5.1 of \cite{GM}). Let us prove the last assertion.

We assume that the neighborhood $U$ is absolutely convex and contains $S(\aaa)$ for some $\aaa=(a_n)_{n\in\NN}\in\AAA$. Choose $q\in\NN$ such that  for every $n\geq q$ it follows that  $ta< a_n $. Hence $\lambda (ta) e_n \in S(\aaa)$ for every $\lambda\in[-1,1]$. Now if $q<m_1<\cdots<m_t$ and $\lambda_1,\dots,\lambda_t\in[-1,1]$, then
\[
v=\frac{1}{t} \cdot \lambda_1 (ta)e_{m_1}+\cdots +\frac{1}{t} \cdot \lambda_t (ta) e_{m_t}
\]
belongs to $U$.
\end{proof}

Below we describe the dual space of $L_G(\mathbf{s})$.
\begin{proposition} \label{p:dual-L(s)}
The dual space $E$ of $L_G(\mathbf{s})$ is linearly isomorphic to the space $c_0$ of all real-valued sequences converging to zero.
\end{proposition}

\begin{proof}
From the description of the base $\mathcal{N}_\mathbf{s}$ of the topology $\pmb{\nu}_{\mathbf{s}}^G$ of $L_G(\mathbf{s})$ given in Proposition \ref{p:Ascoli-Finest}, it immediately follows that $\pmb{\nu}_{\mathbf{s}}^G$ is strictly weaker than the  box topology $\TTT_b$ on $\IR^{(\NN)}$.  It is well known that the dual space of $\big( \IR^{(\NN)}, \TTT_b\big)$ is linearly isomorphic to the direct product $\IR^\NN$. Now the continuity of the identity map $\big( \IR^{(\NN)}, \TTT_b\big) \to L_G(\mathbf{s})$ implies that  $E\subseteq \IR^\NN$. Thus, to prove the proposition we have to show that $\chi=(y_n)\in E$ if and only if $y_n\to 0$.

{\em Claim 1. If $\chi=(y_n)\in E$, then $y_n \to 0$}. Indeed, suppose for a contradiction that there is a sequence $0< r_1<r_2< \dots$ of indices such that $|y_{r_i}| > a$ for some   $a>0$. We assume that all $y_{r_i}$ are positive. Let $U$ be a neighborhood of zero in $L_G(\mathbf{s})$ such that $\chi(U) \subseteq [-1,1]$ and take $t\in\NN$ such that $ta>1$. By Proposition \ref{p:Ascoli-Finest}, there is  $q\in \NN$ such that every vector of the form
\[
v= e_{m_{1}}+ \cdots +  e_{m_{t}}, \; \mbox{ where }\; q<m_1<\cdots<m_t,
\]
belongs to $U$. Take $i\in \NN$ such that $r_i > q$ and set
$
w:=  e_{r_{i+1}} + \cdots +    e_{r_{i+t}}.
$
Then $w\in U$ and
\[
\chi(w)=y_{r_{i+1}} + \cdots +    y_{r_{i+t}} > t\cdot a >1,
\]
a contradiction. Thus $y_n \to 0$.

{\em Claim 2. If $y_n \to 0$, then $\chi\in E$}. Indeed, since $\chi$ is linear, by Proposition \ref{p:Ascoli-Finest}, it is sufficient to show that  there is  a sequence $\mathrm{S}=\{ S(\aaa_k) \}_{k\in\w}\subset \mathcal{S}_\mathbf{s} $ such that $\chi\big( U(\mathrm{S})\big) \subseteq [-1,1]$. To this end, for every $i\in\w$, we shall find $\aaa_i \in \AAA$ such that
\begin{equation} \label{equ:L(s)-character}
\chi\big(S(\aaa_i)\big) \subseteq \left[-\frac{1}{2^{i+1}}, \frac{1}{2^{i+1}} \right].
\end{equation}
Set $I_1 :=\{ n\in\NN: y_n=0\}$ and $I_2 :=\NN\setminus I_1$. If $n\in I_1$, set $a_i(n):= n$, and  set
\[
a_i(n):= \frac{1}{2y_n \cdot 2^{i+1}}, \mbox{ for every } n\in I_2.
\]
It is clear that $a_i(n) \to \infty$ at $n\to\infty$, so $\aaa_i \in \AAA$. The inclusion (\ref{equ:L(s)-character}) holds trivially by the construction of $\aaa_i$ and $S(\aaa_i)$. Now, if $w\in U(\mathrm{S})$ has a decomposition $w=v_0+\cdots +v_k$ with $v_i\in S(\aaa_i)$, we obtain
\[
|\chi(w)| \leq \sum_{i=0}^k |\chi(v_i)| \leq \sum_{i=0}^k \frac{1}{2^{i+1}} <1.
\]
Thus $\chi\in E$.
\end{proof}

The next important result is proved in \cite{HeZu,Smith}, see also \cite[23.32]{HR1}.
\begin{fact} \label{f:Mackey-group-lcs}
Let $E$ be a locally convex space. Then the mapping $p:E' \to \widehat{E}$, defined by the equality
\[
p(f)=\exp\{ 2\pi i f\}, \quad \mbox{ for all } f\in E',
\]
is a group isomorphism between $E'$ and $\widehat{E}$.
\end{fact}

Now Theorem \ref{t:L(s)-Mackey}  immediately follows from (\ref{equ:L(s)=Lg(s)}) and the following result.

\begin{theorem} \label{t:method-L(s)}
The space $L_G(\mathbf{s})$ is not a Mackey space. Consequently $L_G(\mathbf{s})$ is not a Mackey group.
\end{theorem}

\begin{proof}
First we describe a method for constructing of compatible group topologies on $L_G(\mathbf{s})$. Let $h$ be a continuous homomorphism from $\IR$ to a locally quasi-convex  abelian group $H$.  Define an algebraic monomorphism $T_H:\IR^{(\NN)}\to L_G(\mathbf{s})\times \Ff_0(H)$ by
\begin{equation} \label{equ:L(s)-compatible-1}
T_H\big( (x_k)\big):= \bigg( (x_k), \big( h(x_k)\big) \bigg), \quad \forall \; (x_k)\in \IR^{(\NN)},
\end{equation}
and let $\TTT_H$ be the locally quasi-convex group topology on $\IR^{(\NN)}$ induced from $L_G(\mathbf{s})\times \Ff_0(H)$.

{\em Step 1.  The topology $\TTT_H$ is weaker than the box topology $\TTT_b$ on $\IR^{(\NN)}$.}
As we explained in the beginning of the proof of Proposition \ref{p:dual-L(s)}, the topology $\pmb{\nu}_{\mathbf{s}}^G$ is weaker than $\TTT_b$. Therefore we have to show only that the map
\[
p: \big( \IR^{(\NN)}, \TTT_b\big) \to \Ff_0(H), \quad p\big((x_k)\big):=\big( h(x_k)\big),
\]
is continuous at zero. Fix a neighborhood $U$ of the identity in $H$. Choose  a neighborhood $V$ of zero in $\IR$ such that $h(V)\subseteq U$. Set $W:=\IR^{(\NN)} \cap V^\NN$. Then $W$ is a neighborhood of zero in $\big( \IR^{(\NN)}, \TTT_b\big)$ such that for every $(x_k)\in W$ (recall that $(x_k)$ has finite support) we have
\[
p\big((x_k)\big)= \big( h(x_k)\big) \in c_0(H) \cap U^\NN.
\]
Hence $p$ is continuous. Thus $\TTT_H$ is weaker than $\TTT_b$.

\medskip
{\em Step 2. We claim that the topology $\TTT_H$ is compatible with $\pmb{\nu}_{\mathbf{s}}^G$ such that $\pmb{\nu}_{\mathbf{s}}^G\leq\TTT_H$.}  Indeed, set $G:= \big(\IR^{(\NN)}, \TTT_H\big)$. Step 1 implies that
\[
\big( L_G(\mathbf{s})\big)^\wedge  \subseteq G^\wedge \subseteq \big( \IR^{(\NN)}, \TTT_b\big)^\wedge.
\]
Taking into account  Proposition \ref{p:dual-L(s)} and Fact \ref{f:Mackey-group-lcs}, we can identify $G^\wedge$ with a subgroup $Y$ of $\big( \IR^{(\NN)}, \TTT_b\big)' =\IR^\NN$ containing $c_0$. Thus, to prove the claim  we have to show that $Y=c_0$.

Fix arbitrarily $\chi=(y_n)\in Y$. We have to show that $y_n$ converges to zero.
Suppose for a contradiction that $y_n\not\to 0$. So there is a sequence $0< m_1< m_2<\dots$ of indices and $A\in \NN$ such that $|A \cdot y_{m_i}|>1$ for every $i\in\NN$. We assume that all $y_{m_i}$ are positive.
Since $\chi$ is $\TTT_H$-continuous,  there exists a standard neighborhood $W=T_H^{-1}\big(U\times V^\NN\big)$ of zero in $G$, where $U$ is a neighborhood of zero in $L_G(\mathbf{s})$ and $V$ is a neighborhood of the identity in $H$, such that $\chi(W)\subseteq \Ss_+$. Observe that, by (\ref{equ:L(s)-compatible-1}),  $(x_k)\in W$ if and only if
\begin{equation} \label{equ:L(s)-compatible-2}
(x_k)\in U \mbox{ and } h(x_k) \in V \mbox{ for every } k\in \NN.
\end{equation}
Choose $l\in\NN$ such that $h\big([-1/2l,1/2l]\big)\subseteq V$. Set
\[
t:=lA, \; a:=\frac{1}{2l} \; \mbox{ and } \; \lambda_j :=\frac{1}{A\cdot y_{m_{q+j}}} \mbox{ for }  j=1,\dots,t.
\]
Now Proposition \ref{p:Ascoli-Finest} implies that there exists
$q\in\NN$ such that the vector
\[
v:=\lambda_1 \cdot \frac{1}{2l} e_{m_{q+1}}+\cdots + \lambda_t \cdot \frac{1}{2l} e_{m_{q+t}}
\]
belongs to $U$. Since $0<\lambda_j <1$, the choice of $l$ implies that $h(\lambda_j/2l)\in V$ for every $ j=1,\dots,t$. Hence, by  (\ref{equ:L(s)-compatible-2}), we obtain $v\in W$. Therefore
\[
\chi(v) =\exp\left\{ 2\pi i \sum_{j=1}^t y_{m_{q+j}} \cdot \frac{1}{2l  A\cdot y_{m_{q+j}}} \right\} = \exp\left\{ 2\pi i \frac{t}{2l A} \right\} =e^{ \pi i} =-1 \in \Ss_+,
\]
a contradiction. Thus $y_n\to 0$ and $Y=c_0$. The inequality $\pmb{\nu}_{\mathbf{s}}^G\leq\TTT_H$ holds trivially.

{\em Step 3. $L_G(\mathbf{s})$ is not a Mackey space.} Indeed, let $H=\IR$ and $h:\IR\to H$ be the identity map. Then $\Ff_0(\IR)$ is the classical Banach space $c_0$. Therefore the compatible topology $\TTT_H$ is a locally convex vector topology such that $\pmb{\nu}_{\mathbf{s}}^G \leq \TTT_H$. We show that $\pmb{\nu}_{\mathbf{s}}^G \not= \TTT_H$. Indeed, by the definition of $\pmb{\nu}_{\mathbf{s}}^G$, we have $e_n\to 0$ in $\pmb{\nu}_{\mathbf{s}}^G$. However, $e_n\not\to 0$ in the Banach space $c_0$, and hence $e_n\not\to 0$ in $\TTT_H$.

The space $L_G(\mathbf{s})$ is not a Mackey group by (i) of Proposition 2.5 of \cite{Gab-Mackey}.
\end{proof}

%%%%%%%%%%%%%%%%%%%%%%%%%%%%%%
%%%%%%%%%%%%%%%%%%%%%%%%%%%%%%
%%%%%%%%%%%%%%%%%%%%%%%%%%%%%%
%%%%%%%%%%%%%%%%%%%%%%%%%%%%%%
%%%%%%%%%%%%%%%%%%%%%%%%%%%%%%

\bibliographystyle{amsplain}

\begin{thebibliography}{10}


\bibitem{CMPT}
M.J.~Chasco, E.~Mart\'{\i}n-Peinador,  V.~Tarieladze, On Mackey topology for groups, Studia Math. \textbf{132} (1999), 257--284.

\bibitem{DMPT}
D.~Dikranjan, E.~Mart\'{\i}n-Peinador,  V.~Tarieladze, Group valued null sequences and metrizable non-Mackey groups, Forum Math. \textbf{26} (2014), 723--757.

\bibitem{Gab}
S.~Gabriyelyan, Groups of quasi-invariance and the Pontryagin duality,  Topology Appl. \textbf{157} (2010), 2786--2802.

\bibitem{Gab-Mackey}
S.~Gabriyelyan, On the Mackey topology for abelian topological groups and locally convex spaces,  Topology Appl.  \textbf{211} (2016), 11--23.

\bibitem{Gab-Cp}
 S.~Gabriyelyan, A  characterization of barrelledness of $C_p(X)$, J. Math. Anal. Appl. \textbf{439} (2016), 364--369.

\bibitem{GM}
S.S. Gabriyelyan, S.A. Morris,  Free topological vector spaces, Topology Appl., \textbf{223} (2017), 30--49.

\bibitem{Gra}
M.~Graev, Free topological groups, Izv. Akad. Nauk SSSR Ser. Mat. \textbf{12} (1948), 278--324 (In Russian). Topology and Topological Algebra. Translation Series 1, \textbf{8} (1962), 305--364.

\bibitem{HR1}
E.~Hewitt, K.A.~Ross, {\em Abstract Harmonic Analysis}, Vol. I, 2nd ed. Springer-Verlag, Berlin, 1979.

\bibitem{HeZu}
E.~Hewitt, H. Zuckerman, A group-theoretic method  in approximation theory, Ann. Math. \textbf{52} (1950), 557--567.


\bibitem{Jar}
H.~Jarchow, \emph{Locally Convex Spaces}, B.G. Teubner, Stuttgart, 1981.

\bibitem{Kothe}
G.~K\"{o}the, \emph{Topological vector spaces}, Vol. I, Springer-Verlag, Berlin, 1969.

\bibitem{Smith}
M.F.~Smith, The Pontrjagin duality theorem in linear spaces, Ann. Math. \textbf{56} (1952), 248--253.



\end{thebibliography}

\end{document}